\pdfoutput=1
\RequirePackage{ifpdf}
\ifpdf 
\documentclass[pdftex]{sigma}
\else
\documentclass{sigma}
\fi

\usepackage{tikz}
\usetikzlibrary{matrix,arrows}

\numberwithin{equation}{section}

\newtheorem{Theorem}{Theorem}[section]
\newtheorem{Corollary}[Theorem]{Corollary}
\newtheorem{Lemma}[Theorem]{Lemma}
\newtheorem{Proposition}[Theorem]{Proposition}
 { \theoremstyle{definition}
\newtheorem{Definition}[Theorem]{Definition}
\newtheorem{Remark}[Theorem]{Remark} }

\DeclareMathOperator{\Aut}{Aut}
\DeclareMathOperator{\Span}{Span}
\newcommand{\SZ}{\mathbb{Z}} 
\newcommand{\SR}{\mathbb{R}} 
\newcommand{\SC}{\mathbb{C}} 
\newcommand{\SH}{\mathbb{H}} 
\newcommand{\SO}{\mathbb{O}} 

\begin{document}

\allowdisplaybreaks

\newcommand{\arXivNumber}{1811.03613}

\renewcommand{\PaperNumber}{078}

\FirstPageHeading

\ShortArticleName{The Transition Function of $G_2$ over $S^6$}

\ArticleName{The Transition Function of $\boldsymbol{G_2}$ over $\boldsymbol{S^6}$}

\Author{\'Ad\'am GYENGE}

\AuthorNameForHeading{\'A.~Gyenge}

\Address{Mathematical Institute, University of Oxford, UK}
\Email{\href{mailto:Adam.Gyenge@maths.ox.ac.uk}{Adam.Gyenge@maths.ox.ac.uk}}

\ArticleDates{Received May 23, 2019, in final form September 26, 2019; Published online October 09, 2019}

\Abstract{We obtain explicit formulas for the trivialization functions of the ${\rm SU}(3)$ principal bundle $G_2 \to S^6$ over two affine charts. We also calculate the explicit transition function of this fibration over the equator of the six-sphere. In this way we obtain a new proof of the known fact that this fibration corresponds to a generator of $\pi_{5}({\rm SU}(3))$.}


\Classification{57S15; 55R10; 55R25}

\section{Introduction}

The well-known classification of simple Lie groups shows that $G_2$ is the smallest among the exceptional types. Further interesting properties and applications of it are numerous. In this paper we revisit the compact real form of~$G_2$ from the viewpoint of differential geometry.

We identify $G_2$ with $ \Aut \SO \subset {\rm SO}(7)$, the automorphism group of the Cayley octonions. It is a classical fact that there is a fibration $p\colon G_2 \rightarrow S^6$, which makes $G_2$ a locally trivial ${\rm SU}(3)$-bundle over $S^6$. It is also known that the principal ${\rm SU}(3)$-bundles over $S^6$ are classified by $\pi_5({\rm SU}(3))=\SZ $.

A natural question is that to which element in $\pi_5({\rm SU}(3))=\SZ $ does the fibration $G_2 \rightarrow S^6$ correspond? In other words, what is the homotopy class of the transition function $S^5 \rightarrow {\rm SU}(3)$ of the above fibration, where $S^5 \subset S^6$ is the equator of the six-sphere?
\begin{Theorem}[{\cite[Proposition~2]{chaves96complex}}]\label{thm:transfun}
 The class of the transition function of the fibration $G_2 \rightarrow S^6$ is a generator of $\pi_{5}({\rm SU}(3))$.
\end{Theorem}

Let $U_1 = S^6{\setminus}\{ S \}$ and $U_2 = S^6{\setminus}\{ N \}$ be two affine charts on $S^6$, where $S$ and $N$ are the south and north poles. Here we consider $S^6$ as the unit six-sphere
\[ S^6=\bigg\{(0,x_2,\dots,x_8) \in \mathbb{O}\colon \sum_i|x_i|^2=1 \bigg\}\]
in the 7-dimensional vector space of purely imaginary octonions, such that $S=(0,-1,0,\dots,0)$ and $N=(0,1,0,\dots,0)$. Our first result is
an explicit formula for the trivialization functions
\[ \psi_1 \colon \ p^{-1}(U_1) \rightarrow U_1 \times {\rm SU}(3) \]
and
\[ \psi_2 \colon \ p^{-1}(U_2) \rightarrow U_2 \times {\rm SU}(3) \]
deduced in Propositions~\ref{prop:trivn} and \ref{prop:trivs} below.

Using the octonion multiplication the linear subspace $\{x_1=x_2=0\}$ of $\mathbb{O}$ can naturally be identified with $\mathbb{C}^3$. For the precise description of this identification see Section~\ref{section3.2} below. Let $S^5 \subset S^6$ be the equator. This is a~(real) submanifold of $\mathbb{C}^3$ given as
\[ S^5=\big\{(u,v,w) \in \mathbb{C}^3 \colon |u|^2+|v|^2+|w|^2=1 \big\}.\]
The transition function of the ${\rm SU}(3)$-bundle $G_2 \rightarrow S^6$ between the two affine charts of $S^6$ is the ``gluing map'' $t_{12}\colon U_1 \cap U_2 \to {\rm SU}(3)$ for which
\[ \psi_1 \circ \psi_2^{-1}(\xi,\phi)=(\xi, t_{12}(\xi)\phi),\]
where $\xi \in U_1 \cap U_2$ and $\phi \in {\rm SU}(3)$.
Let $\theta:=t_{12}|_{S^5}$ be the restriction of the transition function to the equator. Our second main result is an \emph{explicit} formula for this mapping.
\begin{Theorem} \label{thm:trans}
 The explicit formula of the transition function of the fibration $G_2 \rightarrow S^6$ between the two charts
 \[S^6= ( U_1 ) \cup ( U_2 ) \]
 over the equator $S^5 \subset S^6$ is
 \begin{gather*}
 \theta\colon \ S^5 \rightarrow {\rm SU}(3), \qquad
 \begin{pmatrix}
 u\\ v\\ w
 \end{pmatrix}\mapsto
 \begin{pmatrix}
 u^2 & vu +\overline{w} &wu-\overline{v}\\
 uv-\overline{w} & v^2 & wv+\overline{u} \\
 uw+\overline{v}& vw-\overline{u} & w^2
 \end{pmatrix}.
 \end{gather*}
\end{Theorem}

Our result strengthens \cite[Corollary~3]{chaves96complex}, where this transition function was obtained \emph{up to homotopy}. In particular, our results give a new proof for Theorem~\ref{thm:transfun}. We expect that having an explicit formula for the transition function can be useful in several applications. These may include for example the calculation of Gromov-Witten invariants of $G_2$ and/or $S^6$ as well as calculations about the classifying stack $BG_2$ (see for example \cite[Section~3.2]{pirisi2017motivic}).

It was noted in \cite[Section~2]{puttmann2003presentations} that the map $\theta$ can also be written as
\begin{gather*}
\theta\colon \ S^5 \rightarrow {\rm SU}(3), \qquad z\mapsto zz^t+\overline{M_z},
\end{gather*}
where $z=(u,v,w)^t$ and
\[M_z=
\begin{pmatrix}
0 & w &-v\\
-w & 0 & u \\
v & -u & 0
\end{pmatrix}.
\]
It is standard that $M_z$ is the complexification of the usual cross product on the Euclidean 3-space or, equivalently, of the Spin 1 representation of $\mathfrak{su}(2) \cong \mathbb{R}^3$. A straightforward computation reveals that
\[ zz^tM_z=M_zzz^t=0.\]
This shows in particular that the matrices $zz^t$ and $\overline{M_z}$ are orthogonal with respect to the Frobenius inner product
\[ (A,B) \mapsto \operatorname{Tr}\big(\overline{A^t}B\big) \]
of complex matrices.

The structure of the paper is as follows. In Section~\ref{section2} we give a brief introduction to the algebra of Cayley octonions and to several known facts about the group $G_2$. The new results of the paper are obtained in Section~\ref{section3}.

\section[Some known facts about $G_2$]{Some known facts about $\boldsymbol{G_2}$}\label{section2}

\subsection{Cayley octonions}\label{section2.1}

To perform calculations in the group $G_2$ we collect some known facts about the Cayley algebra of octonions. We follow \cite{postnikov1986lie} where detailed proofs of the following results are given.

Let $\mathcal{A}$ be an algebra over the reals. A linear mapping $a \mapsto \bar{a}$ of $\mathcal{A}$ to itself is said to be a~\textit{conjugation} or \textit{involutory antiautomorphism} if $\bar{\bar{a}}=a$ and $\overline{ab}=\bar{b}\bar{a}$ for any elements $a,b \in A$ (the case $\bar{a}=a$ is not excluded).
\begin{Definition}[Cayley--Dickson construction \cite{baez2002octonions, postnikov1986lie}] Consider the vector space of the direct sum of two copies of an algebra with conjugation: $\mathcal{A}^2=\mathcal{A} \oplus \mathcal{A}$. A multiplication on $\mathcal{A}^2$ is defined as
 \[(a,b)(u,v)=(au-\bar{v}b,b\bar{u}+va).\]
 It is easy to check, that relative to this multiplication the vector space $\mathcal{A}^2$ is an algebra of dimension $2 \cdot \mathrm{dim}(\mathcal{A})$. This is called the \textit{doubling} of the algebra $\mathcal{A}$.
\end{Definition}
\begin{Remark} The correspondence $a \mapsto (a,0)$ is a monomorphism of $\mathcal{A}$ into $\mathcal{A}^2$. Therefore we will identify elements $a$ and $(a,0)$ and thus assume $\mathcal{A}$ is a subalgebra of $\mathcal{A}^2$. If $\mathcal{A}$ has an identity element, then the element $1=(1,0)$ is obviously an identity element in $\mathcal{A}^2$.
\end{Remark}
A distinguished element in $\mathcal{A}^2$ is $e=(0,1)$. It follows from the definition of multiplication that $be=(0,b)$ and hence $(a,b)=a+be$ for all $a,b\in \mathcal{A}$. Thus every element of the algebra $\mathcal{A}^2$ is uniquely written as $a+be$. Moreover, the following identities are true:
\begin{gather}a(be)=(ba)e, \qquad (ae)b=(a\bar{b})e, \qquad (ae)(be)=-\bar{b}a.\label{eq:doublingid}
\end{gather}
In particular $e^2=-1$.

To iterate the Cayley--Dickson construction it is necessary to define a conjugation in $\mathcal{A}^2$. This will be done by the formula
\[\overline{a+be}=\bar{a}-be.\]
This is involutory, $\SR$-linear and is simultaneously an antiautomorphism. It is straightforward to check that if $\mathcal{A}$ is a metric algebra, then $(a+be)\overline{(a+be)} \in \SR$ and it is obviously positive if~$a$ or~$b$ is not~0. Hence, in this case $\mathcal{A}^2$ is also metric algebra.

The doubling $\SR^2$ of the field $\SR$ is the algebra $\SC$ of complex numbers and the doubling~$\SC^2$ of~$\SC$ is the algebra of quaternions $\SH$. In the latter case~$e$ is denoted by $j$ and $ie$ is denoted by~$k$, and thus a general quaternion is of the form $r=r_1+r_2i+r_3j+r_4k$, where $r_i \in \SR$, $i=1,2,3,4$. Due to the second identity of~\eqref{eq:doublingid}, $ea=\bar{a}e$ for all $a \in \mathcal{A}$. Therefore, $\mathcal{A}^2$ is not commutative if the original conjugation is not the identity mapping. In particular $\SH$ is not commutative, as it is well known.

The doubling of the algebra of quaternions leads to an 8-dimensional algebra over the reals.
\begin{Definition} The algebra $\SO=\SH^2$ is the \textit{Cayley algebra}, and its elements are called \textit{octonions} or \textit{Cayley numbers}.
\end{Definition}
By definition every octonion is of the form $\xi=a+be$, where $a$ and $b$ are quaternions. The basis of $\SO$ consists of $1$ and seven elements
\[i,\ j,\ k,\ e,\ f=ie,\ g=je,\ h=ke .\]
The square of each of these elements is $-1$, and they are orthogonal to~$1$. To avoid abusive use of parentheses, both juxtaposition and dots will be used to denote multiplication in $\SO$. The next lemma gives a list of important properties and identities in $\SO$ which we will use to prove our results.

\begin{Lemma}[{\cite{postnikov1986lie}}]\label{lem:octprop}\quad
 \begin{enumerate}\itemsep=0pt
\item[$1.$] The algebra $\SO$ is alternative. That is,
 \[ (a b ) b=a (b b),\qquad a(a b)=(a a)b. \]
\item[$2.$] 
The identity of elasticity $($or flexibility$)$ holds in $\SO$:
 \[ (ab)a=a(ba). \]
\item[$3.$] The algebra $\SO$ is a normed algebra with the norm generated by the metric. In particular, it is a division algebra.
\item[$4.$] 
For all $a,x,y \in \SO$ \[ax\cdot\overline{y}+ay\cdot\overline{x}=2 \langle x,y \rangle a.\]
\item[$5.$] 
For all $a,x,y \in \SO$ \[ax\cdot y+ay\cdot x=a\cdot xy+a\cdot yx. \]
\item[$6.$] 
For all $a,b,x,y \in \SO$ \[ \langle ax,by\rangle+\langle bx,ay\rangle=2\langle a,b\rangle \langle x,y\rangle. \]
\item[$7.$] 
The Moufang identity holds in $\SO$: \[ a (bc) a= (ab)(ca). \]
 \end{enumerate}
\end{Lemma}

\subsection[$G_2$ and the subgroup ${\rm SU}(3)$]{$\boldsymbol{G_2}$ and the subgroup $\boldsymbol{{\rm SU}(3)}$}\label{section2.2}

The group $G_2$ is defined as the automorphism group $\Aut \SO $ of the octonions. It follows from standard facts on unital normed algebras that $ G_2 \subset O(7)$.

Let $\SO' \subset \SO$ be the 7-dimensional subspace of purely imaginary octonions. Consider the subset of the vector space $\SO'$ consisting of elements~$\xi$, such that $|\xi|=1$. This set is a 6-dimensional sphere, which is denoted by~$S^6$. An automorphism $\Phi \colon \SO \rightarrow \SO$ sends the elements $i$, $j$ and~$e$ to elements $\xi=\Phi i$, $\eta= \Phi j$ and $\zeta = \Phi e$ in $S^6$ such that $\eta$ is orthogonal to $\xi$ and $\zeta$ is orthogonal to $\xi$, $\eta$ and $\xi \eta$. The next theorem shows, that these conditions are not only necessary but also sufficient for the existence of the automorphism $\Phi$.

The statement of the following theorem is classical.
\begin{Theorem}[{\cite[p.~309]{postnikov1986lie}}] \label{thm:autcorr}
 For any elements $\xi, \eta, \zeta \in S^6$ such that
 \begin{enumerate}\itemsep=0pt
 \item[$(a)$] $\eta$ is orthogonal to $\xi$,
 \item[$(b)$] $\zeta$ is orthogonal to $\xi$, $\eta$ and $\xi \eta$
 \end{enumerate}
 there is a unique automorphism $\Phi\colon \SO \rightarrow \SO$ for which \[ \xi=\Phi i,\qquad \eta= \Phi j,\qquad \zeta = \Phi e.\]
\end{Theorem}

Let
\[p\colon \ G_2 \rightarrow S^6,\qquad \Phi \mapsto \Phi i\]
be the evaluation mapping on $i$. From Theorem~\ref{thm:autcorr} it follows that the group $G_2 = \Aut \SO $ acts transitively on~$S^6$, i.e., the mapping~$p$ is surjective. Let us denote by $K$ the stabilizer (isotropy) group of $i$ under the action of~$G_2$. Equivalently,
\[K=\{\Phi\colon \SO \rightarrow \SO \, | \, \Phi i = i \}=p^{-1}(i)\]
is the fiber of $p$ over $i$. Due to the standard theorem \cite[Theorem~9.24]{lee2003introduction} of transitive Lie group actions
\[G_2/K \approx S^6.\]

\begin{Lemma} \label{lem:fibsu3}
 There is a canonical isomorphism $K \cong {\rm SU}(3)$.
\end{Lemma}
\begin{proof}
 The subspace $V=\Span \{1,i\}^\perp$ of the algebra $\SO$ is closed under the multiplication by~$i$ and thus it can be considered as a vector space over the field $\SC$ with basis $j$, $e$, $g$. The Hermitian product in $\SO$ induces in~$V$ a Hermitian product with respect to which the basis~$j$,~$e$,~$g$ is orthogonal. Any automorphism $\Phi\colon \SO \rightarrow \SO$ which leaves the element~$i$ fixed, i.e., which is in the subgroup~$K$, defines an operator $V \rightarrow V$ linear over $\SC$. This operator preserves the Hermitian product, and therefore it is an unitary operator.

 The elements of the group ${\rm SU}(3)$ are $3 \times 3$ matrices of the form $[v_1|v_2|v_3]$ consisting of complex orthogonal column vectors having unit length and where $v_3$ is the element in the subspace $\Span_\SC \{v_1,v_2\}^\perp \approx \SC$ such that the determinant of the matrix is~$1$. One can show that the third column is determined by the first two. For a particular $\Phi \in p^{-1}(i)$, the vectors $\eta=\Phi(j)$ and $\zeta=\Phi(e)$ are perpendicular to $i$ and complex orthogonal to each other. Thus, they can be thought as the first and second column of such a matrix and in this case the third column will be $\eta\zeta=\Phi(j)\Phi(e)=\Phi(je)=\Phi(g)$. Combining this with Theorem~\ref{thm:autcorr} it follows that $K$ coincides with ${\rm SU}(3)$.
\end{proof}

As a consequence, we have that ${\rm SU}(3) \subset G_2$ and $G_2 / {\rm SU}(3) \approx S^6$.

\begin{Corollary} Consider the evaluation mapping $p\colon G_2 \rightarrow S^6$, $\Phi \mapsto \Phi i$ defined above. This makes $G_2$ a locally trivial ${\rm SU}(3)$-bundle over~$S^6$.
\end{Corollary}

\subsection{The subgroup of inner automorphisms}\label{section2.3}

In an associative division algebra, such as the quaternions over the reals, the mapping
\[ q_r\colon \ x\mapsto rxr^{-1} \]
is always an automorphism for any invertible element $r$, which is called an inner automorphism. In a non-associative algebra it is not always true that
\[(rx)r^{-1}=r\big(xr^{-1}\big), \qquad \textrm{for all} \quad x,\,r.\]
Moreover, not every invertible element generate an inner automorphism. Still, in the case of the octonions a well defined linear transformation associated with an element $r$ can be defined because of the following lemma.
\begin{Lemma} \label{lem:innautass} For any $r,x \in \SO$
 \[(rx)r^{-1}=r\big(xr^{-1}\big).\]
\end{Lemma}
\begin{proof}If the coordinates of $r$ in the standard basis are $(r_1,\dots,r_8)$, then $r^{-1}=\frac{\bar{r}}{|r|^2}=\frac{2r_1-r}{|r|^2}$. Therefore, using Lemma~\ref{lem:octprop}(2) we have
 \begin{gather*} (rx)r^{-1} =(rx)\frac{2r_1-r}{|r|^2}=\frac{1}{|r|^2}( (rx)2r_1-(rx)r )= \frac{1}{|r|^2}( r(x2r_1)-r(xr) )= r\big(xr^{-1}\big).\tag*{\qed} \end{gather*}
\renewcommand{\qed}{}
\end{proof}

The following result classifies those elements $r$ for which the linear map $q_r$ is an automorphism of $\SO$. For completeness, we reproduce its original proof.
\begin{Theorem}[\cite{lamont1963arithmetics}] \label{thm:reoctin}
 A non-real octonion $r$ with coordinates $(r_1,\dots,r_8)$ induces an inner automorphism of $\SO$ if and only if $4r_1^2=|r|^2$.
\end{Theorem}
\begin{proof} From Lemma~\ref{lem:octprop}(7) for $a=r$, $b=xr^{-1}$ and $c=ryr$ it follows that
 \begin{gather} \big(rxr^{-1}\big)(ryr \cdot r)=r\big(xr^{-1} \cdot ryr \big)r. \label{eq:innaut1}\end{gather}
 Similarly,
 \[\overline{ryr}=\bar{r}\bar{y}\bar{r}=\bar{r}\big(\bar{y} \big(x^{-1}x\big)\big)\bar{r}=\bar{r}\big(\big(\bar{y} x^{-1}\big)x\big)\bar{r}=\big(\bar{r} \cdot \bar{y}x^{-1}\big)(x\bar{r})\]
 and therefore
 \begin{align*}
 ryr &= \overline{\big(\bar{r} \cdot \bar{y}x^{-1}\big)(x\bar{r})}=(\overline{x\bar{r}}) \big(\overline{\bar{r}\cdot \bar{y}x^{-1}}\big)=(r\bar{x})\big(\overline{\bar{y}x^{-1}} \cdot r\big) \\
 &=(r\bar{x}) \big(\overline{x}^{-1}y\cdot r\big)= \big(r \big(|x|^2 x^{-1}\big) \big) \left(\frac{x}{|x|^2}y\cdot r\right)=\big(r x^{-1} \big) (xy\cdot r).
 \end{align*}
 Substituting this into \eqref{eq:innaut1} leads us to
 \begin{align*}
 \big(rxr^{-1}\big)\big(ryr \cdot r\big)&=r\big(xr^{-1} \cdot \big(r x^{-1} \big) (xy\cdot r) \big)r =r\big( \big(\underbrace{xr^{-1}}_a \cdot \underbrace{r x^{-1}}_{a^{-1}} \big) \cdot (xy\cdot r) \big)r\\
 &=r((xy\cdot r) )r= r(xy)r^2,
 \end{align*}
 i.e.,
 \begin{gather} \big(rxr^{-1}\big)\big(ryr^{-1} \cdot r^3\big)=r(xy)r^{-1}\cdot r^3 \label{eq:innautc1}\end{gather}
 for all $x,y,r\in\SO$.

 The mapping $q_r\colon x\mapsto rxr^{-1}$ is an automorphism if and only if
 \[ \big(rxr^{-1}\big)\big(ryr^{-1}\big)=r(xy)r^{-1}.\]
 Multiplying this with $r^3$ from the right we get
 \begin{gather} \big(rxr^{-1}\big)\big(ryr^{-1}\big) \cdot r^3=r(xy)r^{-1}\cdot r^3. \label{eq:innautc2}\end{gather}
 Comparing~\eqref{eq:innautc1} with~\eqref{eq:innautc2} we see that in order for~$q_r$ to be an automorphism~$r^{3}$ must be a~scalar.

 Using the fact that $\bar{r}=2r_1-r$, one has $|r|^2=r\bar{r}=r(2r_1-r)=2rr_1-r^2$ for all $r \in \SO$. Multiplying with $r$ and applying the same equation again we get that
 \[ r^3-2r_1r^2+|r|^2r=r^3-4r_1^2 r+2r_1|r|^2+r|r|^2=0, \]
 and thus
 \[ r^3+2r_1|r|^2=r\big(4r_1^2 -|r|^2\big). \]
 Suppose $r^{3}$ is a scalar. Then each term on the left side is real and therefore either $r$ should be real, or $\big(4r_1^2 -|r|^2\big)$ should be zero. The latter case means that $4r_1^2 =|r|^2$.
\end{proof}

\section[$G_2$ as an ${\rm SU}(3)$-bundle over $S^6$]{$\boldsymbol{G_2}$ as an $\boldsymbol{{\rm SU}(3)}$-bundle over $\boldsymbol{S^6}$}\label{section3}

\subsection{The trivialization functions}\label{section3.1}

Our aim is to determine the transition function of the fibration
\[p\colon \ G_2 \rightarrow S^6, \qquad \Phi \mapsto \Phi i\]
between two charts of $S^6$ given by $U_1=S^6 {\setminus} \{S\} $ and $U_2=S^6 {\setminus} \{N\}$, where $S=-i=(0,-1,0,\dots,0)$ and $N=i=(0,1,0,\dots,0)$. The preimage of $i$ is the set $p^{-1}(i)=\{(i,\eta,\zeta)\colon \eta \perp i, \,\zeta \perp \Span \{i,\eta, i \eta \} \}$. As mentioned above this is isomorphic to ${\rm SU}(3)$ and this isomorphism will be called $\theta_i$.

For any $\xi \in S^6$ let us denote by $V_\xi$ or $T_\xi S^6$ the tangent space (of orthogonal vectors) to $\xi$.
By the considerations above elements in $p^{-1}(i)$ can be considered either as orthonormal vector triples in $V_i=T_iS^6$ or as operators that leave the vector $i$ fixed.
It also follows from the result above that there is a complex structure
\[J_i\colon \ V_i \rightarrow V_i, \qquad v \mapsto iv.\]
This is clearly a mapping from $V_i$ to itself such that $J_{i}^2(v)=i^2v=-v$ for all $v \in V$. Thus, there is an isomorphism $\theta_{i}\colon V_i \rightarrow \SC^3$ that assigns to each operator $\Phi \in p^{-1}(i)$, $\Phi\colon V_i \rightarrow V_i$ its matrix representation in the complex basis $\{j,e,g\}$.

Similarly, $p^{-1}(\xi)=\{(\xi,\eta,\zeta)\colon \eta \perp \xi, \,\zeta \perp \Span \{\xi,\eta, \xi \eta \} \}$ for any $\xi \in S^6$. Any map $\varphi \in p^{-1}(\xi)$ carries $V_i$ to $V_{\xi}$. Again, there is a complex structure on $V_\xi$ denoted by $J_\xi$, which comes from octonion multiplication: $J_\xi(v)=\xi v$. By choosing a complex orthonormal basis in this subspace we give an identification $V_\xi \approx \SC^3$. These considerations imply the following classical result.

\begin{Corollary}[{\cite{ehresmann1950varietes}}] The complex structure given by $J_\xi\colon V_\xi \rightarrow V_\xi$, $v \mapsto \xi v$ defines a smooth almost complex structure $J\colon TS^6 \rightarrow TS^6$, $(\xi,v) \mapsto (\xi, J_\xi(v))$.
\end{Corollary}

 This almost complex structure has the following remarkable property.
\begin{Proposition} A rotation $g\colon S^6 \rightarrow S^6$ is an element of~$G_2$ if and only if its pushforward $g_\ast \colon TS^6 \rightarrow TS^6$, $(x,v) \mapsto (g(x),g(v))$ is $J$-equivariant $($where $J$ is considered as a $\SZ_4$-action on~$TS^6)$, or, in other words, if the following diagram is commutative:
 \begin{center}
 \begin{tikzpicture}[description/.style={fill=white,inner sep=2pt}]
 \matrix (m) [matrix of math nodes, row sep=2.5em,
 column sep=2.2em, text height=2ex, text depth=0.25ex]
 { TS^6 & & TS^6 \\
 TS^6& & TS^6. \\ };
 \path[->] 
 (m-1-1) edge node[auto] {$ J$} (m-1-3)
 edge node[auto,swap] {$ g_\ast $} (m-2-1)
 (m-1-3) edge node[auto] {$ g_\ast $} (m-2-3)
 (m-2-1) edge node[auto] {$ J $} (m-2-3);
 \end{tikzpicture}
 \end{center}
\end{Proposition}
\begin{proof} Because $G_2 \subset O(7)$, any $g \in G_2$ preserves the scalar product. Therefore, $g(V_\xi)=V_{g(\xi)}$ and we need only to prove that the following diagram commutes for all $\xi \in S^6$:
 \begin{center}
 \begin{tikzpicture}[description/.style={fill=white,inner sep=2pt}]
 \matrix (m) [matrix of math nodes, row sep=2.5em,
 column sep=2.2em, text height=2ex, text depth=0.25ex]
 { T_\xi S^6 & & T_\xi S^6 \\
 T_{g(\xi)} S^6& & T_{g(\xi)}S^6. \\ };
 \path[->] 
 (m-1-1) edge node[auto] {$ J_\xi$} (m-1-3)
 edge node[auto,swap] {$ g_\ast $} (m-2-1)
 (m-1-3) edge node[auto] {$ g_\ast $} (m-2-3)
 (m-2-1) edge node[auto] {$ J_{g(\xi)} $} (m-2-3);
 \end{tikzpicture}
 \end{center}
 Since $g\in \Aut \SO$ we have that
 \[ g(J_\xi(\eta))=g(\xi\eta)=g(\xi)g(\eta)=J_{g(\xi)}(g(\eta)),\]
 for all $\xi \in S^6$, $\eta \in V_\xi$.

 Conversely, assume $\xi \in S^6$, $\eta \in \SO'$. Decompose $\eta$ to $\eta_1+\eta_2$ where $\eta_1 \perp \xi$. Suppose $g_\ast$ commutes with $J$. Then
 \[ g(\xi\eta_1)=g(J_\xi(\eta_1))=J_{g(\xi)}(g(\eta_1))=g(\xi)g(\eta_1) ,\]
 and obviously $g(\xi\eta_2)=g(\xi)g(\eta_2)$. Thus, $g(\xi\eta)=g(\xi)g(\eta)$.
\end{proof}

\begin{Proposition} \label{prop:trivn} The trivialization map over $U_1$ is given by
 \[\psi_1\colon \ p^{-1}(U_1) \rightarrow U_1 \times {\rm SU}(3), \qquad \varphi \mapsto (\varphi(i),\theta_{\varphi(i)}(\varphi) ),\]
 where $\varphi(i)$ is the image of $i$ under $\varphi$ and $\theta_{\varphi(i)}(\varphi)$ is given by~\eqref{eq:thetaxi} below.
\end{Proposition}
\begin{proof}In the proof of Lemma~\ref{lem:fibsu3} it was shown, that
 for a particular $\Phi \in p^{-1}(i)$ the vectors $\eta=\Phi(j)$ and $\zeta=\Phi(e)$ are perpendicular to $i$ and complex orthogonal to each other. Thus, they can be thought as the first and second columns of a matrix in ${\rm SU}(3)$ with the third column $\eta\zeta=\Phi(j)\Phi(e)=\Phi(je)=\Phi(g)$. 
 If the coordinates of the vectors are $\eta=(0,y_2,\dots,y_8)$, $\zeta=(0,z_2,\dots,z_8)$ and $\eta \zeta = (0,u_2,\dots,u_8)$, then since $\eta,\zeta,\eta\zeta \in V_i$ we have that $y_2=0$, $z_2=0$ and $u_2=0$. The mapping $\theta_{i}$ is then the following:
 \[\theta_i\colon \ p^{-1}(i) \rightarrow {\rm SU}(3), \qquad (i,\eta,\zeta) \mapsto
 \begin{pmatrix}
 y_3+Iy_4 & z_3+Iz_4 & u_3+Iu_4\\
 y_5+Iy_6 & z_5+Iz_6 & u_5+Iu_6\\
 y_7+Iy_8 & z_7+Iz_8 & u_7+Iu_8
 \end{pmatrix}.\]
 Here $I$ is the imaginary unit in the field $\SC^3$ and not the octonion $i$.
 It follows that
 \[ (i,\eta,\zeta) \mapsto
 \begin{pmatrix}
 \langle \eta,j \rangle + I\langle \eta,k \rangle & \langle \zeta,j \rangle + I\langle \zeta,k \rangle & \langle \eta \zeta,j \rangle + I\langle \eta \zeta,k \rangle\\
 \langle \eta,e \rangle +I\langle \eta,f \rangle & \langle \zeta,e \rangle +I\langle \zeta,f \rangle & \langle \eta \zeta,e \rangle +I\langle \eta \zeta,f \rangle\\
 \langle \eta,g \rangle +I\langle \eta,h \rangle & \langle \zeta,g \rangle +I\langle \zeta,h \rangle & \langle \eta \zeta,g \rangle +I\langle \eta \zeta,h \rangle
 \end{pmatrix},
 \]
 i.e., we represent $\eta,\zeta,\eta\zeta \in V_i$, the images of $j$, $e$ and $g$ in the complex basis $\{j,e,g\}$.

As a consequence, for any $\xi \in U_1 $ and any $\varphi \in p^{-1}(\xi)$, $\varphi$ restricts to a mapping $V_i \rightarrow V_\xi$, which is complex linear, unitary and has determinant~1. We will choose a complex orthonormal basis in $V_\xi$ and write the images of $j$, $e$ and $g$ in this basis. That is, we choose particular identifications $V_i \approx \SC^3$, $V_\xi \approx \SC^3$ and we define $\theta_\xi \colon p^{-1}(\xi) \rightarrow {\rm SU}(3)$ by assigning to each automorphism $\varphi \in p^{-1}(\xi)$ the matrix of the mapping $\varphi\colon \SC^3 \rightarrow \SC^3$. To find a basis in $V_\xi$ we will define a translating automorphism $Q_\xi$ such that $Q_\xi(i)=\xi$. Then, for $a=Q_\xi(j)$, $b=Q_\xi(e)$ and $c=Q_\xi(g)$ the set of vectors $\{a,b,c\}$ is a complex orthonormal basis in $V_\xi$ with respect to the complex structure $J_\xi(v)=\xi v$. Particularly,
 \begin{gather}
\theta_\xi \colon \ p^{-1}(\xi) \rightarrow {\rm SU}(3),\nonumber\\
 (\xi,\eta,\zeta) \mapsto
 \begin{pmatrix}
 \langle \eta,a \rangle + I\langle \eta,J_\xi(a) \rangle & \langle \zeta,a \rangle + I\langle \zeta,J_\xi(a) \rangle & \langle \eta \zeta,a \rangle + I\langle \eta \zeta,J_\xi(a) \rangle\\
 \langle \eta,b \rangle +I\langle \eta,J_\xi(b) \rangle & \langle \zeta,b \rangle +I\langle \zeta,J_\xi(b) \rangle & \langle \eta \zeta,b \rangle +I\langle \eta \zeta,J_\xi(b) \rangle\\
 \langle \eta,c \rangle +I\langle \eta,J_\xi(c) \rangle & \langle \zeta,c \rangle +I\langle \zeta,J_\xi(c) \rangle & \langle \eta \zeta,c \rangle +I\langle \eta \zeta,J_\xi(c) \rangle
 \end{pmatrix}.\label{eq:thetaxi}
\end{gather}
 Using this the trivializing map is given by
 \begin{gather*} \psi_1\colon \ p^{-1}(U_1) \rightarrow U_1 \times {\rm SU}(3), \qquad \varphi \mapsto (\varphi(i),\theta_{\varphi(i)}(\varphi) ) .\tag*{\qed}\end{gather*}
\renewcommand{\qed}{}
\end{proof}

Completely analogously the preimage of $-i$ under the evaluation map $p$ is diffeomorphic to~${\rm SU}(3)$, and in this case the complex structure on $V_{-i}$ is given by $J_{-i}(v)=-iv$. Therefore, $\tilde{\theta}_{-i}$ is defined as
\[ (-i,\eta,\zeta) \mapsto
\begin{pmatrix}
\langle \eta,j \rangle + I\langle \eta,-k \rangle & \langle \zeta,j \rangle + I\langle \zeta,-k \rangle & \langle \eta \zeta,j \rangle + I\langle \eta \zeta,-k \rangle\\
\langle \eta,e \rangle +I\langle \eta,-f \rangle & \langle \zeta,e \rangle +I\langle \zeta,-f \rangle & \langle \eta \zeta,e \rangle +I\langle \eta \zeta,-f \rangle\\
\langle \eta,g \rangle +I\langle \eta,-h \rangle & \langle \zeta,g \rangle +I\langle \zeta,-h \rangle & \langle \eta \zeta,g \rangle +I\langle \eta \zeta,-h \rangle
\end{pmatrix}.
\]

As we did in the previous case, for a general $\xi \in U_2=S^6{\setminus}\{ N \}$ we will choose a translating auto\-morphism $\tilde{Q}_\xi$ with the property that $\tilde{Q}_\xi(-i)=\xi$ implying that $\tilde{Q}_\xi (j),\tilde{Q}_\xi (e),\tilde{Q}_\xi (g) \in V_\xi$ form a complex orthonormal basis. Then we define $\tilde{\theta}_\xi \colon p^{-1}(\xi) \rightarrow {\rm SU}(3)$ by assigning to $\varphi \in p^{-1}(v)$ the matrix of the corresponding linear mapping from $V_{-i}$ onto $V_{\xi}$ written in the bases $\{j,e,g\}$ at $V_{-i}$ and $\big\{\tilde{a},\tilde{b},\tilde{c} \big\} := \big\{\tilde{Q}_\xi (j),\tilde{Q}_\xi (e),\tilde{Q}_\xi (g)\big\}$ at $V_\xi$. Similarly as in the proof Proposition~\ref{prop:trivn} we obtain the following morphism
\begin{gather}
\tilde{\theta}_\xi \colon \ p^{-1}(\xi) \rightarrow {\rm SU}(3), \nonumber\\
(\xi,\eta,\zeta) \mapsto
\begin{pmatrix}
\langle \eta,\tilde{a} \rangle + I\langle \eta,J_\xi(\tilde{a}) \rangle & \langle \zeta,\tilde{a} \rangle + I\langle \zeta,J_\xi(\tilde{a}) \rangle & \langle \eta \zeta,\tilde{a} \rangle + I\langle \eta \zeta,J_\xi(\tilde{a}) \rangle\\
\langle \eta,\tilde{b} \rangle +I\langle \eta,J_\xi(\tilde{b}) \rangle & \langle \zeta,\tilde{b} \rangle +I\langle \zeta,J_\xi(\tilde{b}) \rangle & \langle \eta \zeta,\tilde{b} \rangle +I\langle \eta \zeta,J_\xi(\tilde{b}) \rangle\\
\langle \eta,\tilde{c} \rangle +I\langle \eta,J_\xi(\tilde{c}) \rangle & \langle \zeta,\tilde{c} \rangle +I\langle \zeta,J_\xi(\tilde{c}) \rangle & \langle \eta \zeta,\tilde{c} \rangle +I\langle \eta \zeta,J_\xi(\tilde{c}) \rangle
\end{pmatrix}.\label{eq:thetaxi2}
\end{gather}
As a consequence, the analogue of Proposition~\ref{prop:trivn} is true for this chart.
\begin{Proposition} \label{prop:trivs} The trivialization map over $U_2$ is then given by
 \[\psi_2\colon \ p^{-1}(U_2) \rightarrow U_2 \times {\rm SU}(3), \qquad \varphi \mapsto (\varphi(i),\tilde{\theta}_{\varphi(i)}(\varphi) ),\]
 where $\varphi(i)$ is the image of $i$ under~$\varphi$ and $\tilde{\theta}_{\varphi(i)}(\varphi)$ is given by~\eqref{eq:thetaxi2}.
\end{Proposition}

To summarize, if $Q_\xi ,\tilde{Q}_\xi \in G_2$ are known as functions depending differentiably on $\xi$ with the property that $Q_\xi(i)=\xi$ and $\tilde{Q}_\xi(-i)=\xi$, then an appropriate basis in $V_\xi$ is $a=Q_\xi(j)$, $b=Q_\xi(e)$, $c=Q_\xi(g)$, which are the translations of the basis $j$, $e$, $g$ from~$V_i$ in the case of the first chart. In the case of the second chart~$\tilde{Q}_\xi$ translates $j$, $e$, $g$ from $V_{-i}$ to $V_\xi$. Thus, we need to find elements $Q_\xi \in G_2$ and $\tilde{Q}_\xi \in G_2$. Knowing the first one is enough, because then the second is given due to the identities $Q_{-\xi}(-i)=Q_{-\xi}((-1)i)=Q_{-\xi}(-1)Q_{-\xi}(i)=-1(-\xi)=\xi$.

It will be convenient to look for $Q_\xi$ in the form of an inner automorphism generated by an element $r \in \SO$. The easiest is to look for a unit length octonion that induces $Q_\xi$. For a unit length octonion~$r$ the conjugate of~$i$ with $r$ is
\begin{gather*} ri\bar{r}=\big(0,r_1^2+r_2^2-r_3^2-r_4^2-r_5^2-r_6^2-r_7^2-r_8^2,2(r_2r_3+r_1r_4),2(r_2r_4-r_1r_3),\\
\hphantom{ri\bar{r}=\big(}{} 2(r_2r_5+r_1r_6),2(r_2r_6-r_1r_5),2(r_2r_7-r_1r_8),2(r_1r_7+r_2r_8)\big).\end{gather*}
Since $r_\xi i\bar{r}_\xi=\xi=(0,x_2,\dots,x_8)$ is needed, the following system of equations is to be solved
\begin{gather*}
r_1^2+r_2^2-r_3^2-r_4^2-r_5^2-r_6^2-r_7^2-r_8^2 = x_2,\\
2(r_2r_3+r_1r_4) = x_3, \\
2(r_2r_4-r_1r_3) = x_4, \\
2(r_2r_5+r_1r_6) = x_5, \\
2(r_2r_6-r_1r_5) = x_6, \\
2(r_2r_7-r_1r_8) = x_7, \\
2(r_1r_7+r_2r_8) = x_8.
\end{gather*}
From Theorem~\ref{thm:reoctin} it follows that $r_1=\frac{1}{2}$ is required. The general solution for a fixed $\xi \in U_1$ of this system of equations is
\begin{gather}
r_\xi= \frac{1}{2}\left(1, \sqrt{1+2x_2},\frac{x_3\sqrt{1+2x_2}-x_4}{1+x_2},\frac{x_3+x_4\sqrt{1+2x_2}}{1+x_2}, \right.\nonumber\\
\left.\hphantom{r_\xi= \frac{1}{2}}{}\frac{x_5\sqrt{1+2x_2}-x_6}{1+x_2},\frac{x_5+x_6\sqrt{1+2x_2}}{1+x_2},\frac{x_7\sqrt{1 + 2 x_2} + x_8}{1+x_2},\frac{-x_7 + x_8\sqrt{1 + 2 x_2}}{1+x_2}\right).\label{eq:octhopfsol}
\end{gather}

\subsection{The transition function over the equator}\label{section3.2}
As in the previous sections we cover the base space $S^6$ with two trivializing charts given by $U_1=S^6 {\setminus} \{S\} $ and $U_2=S^6 {\setminus} \{N\}$.
We are interested in the transition function between the two trivializations over the equator. This is enough to reconstruct the whole fibration, since the equator is a deformation retract of the intersection of the charts. The equator $S^5$ will be identified with a submanifold of $\mathbb{C}^3=V_i=V_{-i}$ as
\[ S^5=\big\{(u,v,w) \in \mathbb{C}^3 \,|\, |u|^2+|v|^2+|w|^2=1\big\}, \]
where the coordinate functions $u$, $v$ and $w$ are the duals of $j$, $e$ and $g$ respectively. We are now ready to prove Theorem~\ref{thm:trans} which we restate here.

\begin{Theorem} \label{thm:traneq} The transition function between the two trivializations of the principal ${\rm SU}(3)$-bundle $G_2 \rightarrow S^6$ at the equator is
 \begin{gather*}
 \theta\colon \ S^5 \rightarrow {\rm SU}(3), \qquad
 \begin{pmatrix}
 u\\ v\\ w
 \end{pmatrix}\mapsto
 \begin{pmatrix}
 u^2 & vu +\overline{w} &wu-\overline{v}\\
 uv-\overline{w} & v^2 & wv+\overline{u} \\
 uw+\overline{v}& vw-\overline{u} & w^2
 \end{pmatrix}.
 \end{gather*}
\end{Theorem}

From now on we assume that any $\xi \in \SO$ is in the equator of $S^6$, and thus $x_2=0$. In this case the solution~\eqref{eq:octhopfsol} simplifies to
\[
r_\xi= \frac{1}{2} (1, 1,x_3-x_4,x_3+x_4,x_5-x_6,x_5+x_6,x_7 + x_8,-x_7 + x_8 ).
\]
Due to the fact that $i\xi=(0, 0, -x_4, x_3, -x_6, x_5, x_8, -x_7)$ we have
\[r_\xi=\frac{1}{2}+\frac{i}{2}+\frac{\xi+i\xi}{2}=\frac{(1+i)(1+\xi)}{2}.\]
It is easy to check that $r_\xi$ is really a solution, because in this case due to Lemmas~\ref{lem:octprop}(2) and~\ref{lem:innautass} we may perform the multiplication in arbitrary order:
\begin{gather*}
\frac{(1+i)(1+\xi)}{2}\cdot i \cdot \overline{\frac{(1+i)(1+\xi)}{2}} =\frac{1}{4} (1+i) ((1+\xi)i(1-\xi))(1-i) \\
=\frac{1}{4} (1+i)(i+\xi i-i\xi-\xi i \xi)(1-i) = \frac{1}{4} (1+i)\big(i+2\xi i+ i \xi^2\big)(1-i)=\frac{1}{4} (1+i)2\xi i(1-i) \\
=\frac{1}{4}\big(2\xi i + 2 i\xi i -2\xi i^2 - 2 i\xi i^2\big)=\frac{1}{4}\big(2\xi i -2 i^2 \xi +2\xi +2 i \xi \big)=\frac{1}{4}(2\xi i + 4\xi -2\xi i)=\frac{4\xi}{4}=\xi.
\end{gather*}
Consequently, the required automorphisms for an arbitrary $\xi \in U_1 \cap U_2$ are
\begin{gather*} Q_\xi\colon \ \SO \rightarrow \SO,\qquad x \mapsto r_\xi x\bar{r}_\xi, \\
 \tilde{Q}_\xi\colon \ \SO \rightarrow \SO,\qquad x \mapsto r_{-\xi} x\bar{r}_{-\xi}. \end{gather*}

Once again, the transition function between the two trivializations is
\[\psi_1 \circ \psi_2^{-1}\colon \ U_1 \cap U_2 \times {\rm SU}(3) \rightarrow U_1 \cap U_2\times {\rm SU}(3), \qquad (\xi,\phi) \mapsto \big(\xi, \theta_\xi \circ \tilde{\theta}_\xi^{-1}(\phi) \big).\]
As it was discussed above, the meaning of $\psi_1$ is the following:
\[ (\xi,\eta,\zeta) \mapsto
\begin{pmatrix}
\langle \eta,Q_\xi j \rangle + I\langle \eta, Q_\xi k \rangle & \langle \zeta, Q_\xi j \rangle + I\langle \zeta, Q_\xi k \rangle & \langle \eta \zeta, Q_\xi j \rangle + I\langle \eta \zeta,Q_\xi k \rangle\\
\langle \eta,Q_\xi e \rangle +I\langle \eta,Q_\xi f \rangle & \langle \zeta,Q_\xi e \rangle +I\langle \zeta,Q_\xi f \rangle & \langle \eta \zeta,Q_\xi e \rangle +I\langle \eta \zeta,Q_\xi f \rangle\\
\langle \eta,Q_\xi g \rangle +I\langle \eta,Q_\xi h \rangle & \langle \zeta,Q_\xi g \rangle +I\langle \zeta,Q_\xi h \rangle & \langle \eta \zeta,Q_\xi g \rangle +I\langle \eta \zeta,Q_\xi h \rangle
\end{pmatrix}.
\]
Similarly, $\psi_2$ is
\[ (\xi,\eta,\zeta) \mapsto
\begin{pmatrix}
\langle \eta,\tilde{Q}_\xi j \rangle + I\langle \eta, \tilde{Q}_\xi k \rangle & \langle \zeta, \tilde{Q}_\xi j \rangle + I\langle \zeta, \tilde{Q}_\xi k \rangle & \langle \eta \zeta, \tilde{Q}_\xi j \rangle + I\langle \eta \zeta,\tilde{Q}_\xi k \rangle\\
\langle \eta,\tilde{Q}_\xi e \rangle +I\langle \eta,\tilde{Q}_\xi f \rangle & \langle \zeta,\tilde{Q}_\xi e \rangle +I\langle \zeta,\tilde{Q}_\xi f \rangle & \langle \eta \zeta,\tilde{Q}_\xi e \rangle +I\langle \eta \zeta,\tilde{Q}_\xi f \rangle\\
\langle \eta,\tilde{Q}_\xi g \rangle +I\langle \eta,\tilde{Q}_\xi h \rangle & \langle \zeta,\tilde{Q}_\xi g \rangle +I\langle \zeta,\tilde{Q}_\xi h \rangle & \langle \eta \zeta,\tilde{Q}_\xi g \rangle +I\langle \eta \zeta,\tilde{Q}_\xi h \rangle
\end{pmatrix}.
\]
The mapping $Q_\xi(v)= r_\xi v\bar{r}_\xi$ is linear in $v$, because $\SO$ is distributive and scalars commute with everything. Due to the construction $Q_\xi(x)$ maps the subspace $V_i$ to $V_\xi$ isomorphically.
\begin{Lemma} \label{lem:gxiv} If $v, \xi \in V_i$, then
 \[ Q_\xi(v)= \frac{1}{2}((-1+i+\xi+i\xi) v + \langle v, \xi + i\xi \rangle (1+i+\xi+i\xi)). \]
\end{Lemma}

\begin{proof} To compute $Q_\xi(v)$, four groups of identities will be necessary.

(i) According to the definition of the scalar product in $\SO$ and Lemma~\ref{lem:octprop}(4)
 \begin{gather*}
 v \cdot i\xi=-\overline{i\xi} \cdot \overline{v}+2\langle v, \overline{i\xi} \rangle = \xi i \cdot v + 2\langle v, \xi i \rangle = -i \xi \cdot v - 2\langle v, i \xi \rangle,\\
 iv\cdot \xi= -i\overline{\xi}\cdot \overline{v}+2 \langle v, \overline{\xi} \rangle i= -i\xi\cdot v - 2 \langle v, \xi \rangle i,\\
 \xi v \cdot i=-\xi \overline{i}\cdot \overline{v}+2 \underbrace{\langle i, \overline{v} \rangle}_0 \xi=i \xi \cdot v.
 \end{gather*}
 Therefore
 \begin{gather}
 i\xi \cdot v-\xi v \cdot i- iv\cdot \xi-v\cdot i\xi=2i\xi\cdot v +2 \langle v,\xi \rangle i+2\langle v,i \xi \rangle. \label{eq:gxi1}
 \end{gather}

(ii) Similarly,
 \begin{gather*}
 iv\cdot i\xi =-(i\cdot \overline{i\xi})\overline{v}+2\langle i\xi,\overline{v}\rangle i=(i\cdot \xi i)v+2\langle \xi i,v \rangle i=\xi v +2\langle \xi i,v \rangle i,\\
 (i\xi\cdot v)i =-(i\xi\cdot\overline{i})\overline{v}+2\underbrace{\langle v, \overline{i} \rangle}_0 i\xi=-(i\xi i)v=(ii\xi)v=-\xi v.
 \end{gather*}
 Summing over the two equations this leads to
 \begin{gather}
 iv\cdot i\xi+(i\xi\cdot v)i=\xi v +2\langle \xi i,v \rangle i-\xi v=2\langle \xi i,v \rangle i. \label{eq:gxi2}
 \end{gather}

(iii) With essentially the same tricks one obtains
 \begin{gather*}
 (i\xi \cdot v)\xi=-(i\xi \cdot \overline{\xi})\overline{v}+2 \langle v,\overline{\xi}\rangle i\xi=iv-2\langle v ,\xi \rangle i\xi,\\
 \xi v\cdot i\xi=-\xi \overline{i\xi}\cdot \overline{v}+2 \langle\overline{v},i\xi\rangle\xi=-(\xi\cdot \overline{x}\overline{i})\overline{v}-2\langle v,i\xi\rangle\xi =-iv-2\langle v,i\xi \rangle\xi.
 \end{gather*}
 Therefore
 \begin{gather}
 (i\xi\cdot v)\xi+\xi v\cdot i\xi=iv-2\langle v ,\xi \rangle i\xi-iv-2\langle v,i\xi \rangle\xi=-2\langle v ,\xi \rangle i\xi-2\langle v,i\xi \rangle\xi. \label{eq:gxi3}
 \end{gather}

(iv) Once again,
 \begin{gather}
 \xi v \xi =-\xi\overline{\xi}\cdot\overline{v}+2\langle\xi,\overline{v}\rangle\xi=v-2\langle\xi,v\rangle\xi, \label{eq:gxi4}\\
 i\xi \cdot v\cdot i\xi =-(i\xi)\overline{i\xi}\cdot\overline{v}+2\langle i\xi,\overline{v}\rangle i\xi=v-2\langle i\xi,v\rangle i\xi. \label{eq:gxi5}
 \end{gather}

 Putting these together,
 \begin{gather*}
 Q_\xi(v) = r_\xi v\bar{r}_\xi=\frac{1}{4}(1+i+\xi+i\xi)v(1-i-\xi-i\xi)\\
 \hphantom{Q_\xi(v)}{} = \frac{1}{4}(v+iv+\xi v+i\xi \cdot v)(1-i-\xi-i\xi)\\
\hphantom{Q_\xi(v)}{}= \frac{1}{4}(v+iv+\xi v+i\xi\cdot v -v i-ivi-\xi v\cdot i-(i\xi\cdot v)i\\
\hphantom{Q_\xi(v)=}{} -v\xi-iv\cdot\xi-\xi v\xi-(i\xi\cdot v)\xi-v\cdot i\xi-iv\cdot i\xi-\xi v\cdot i\xi-i\xi\cdot v\cdot i\xi)\\
\hphantom{Q_\xi(v)}{}= \frac{1}{4}(2iv+\xi v-v\xi+(i\xi \cdot v-\xi v \cdot i- iv\cdot \xi-v\cdot i\xi) \\
\hphantom{Q_\xi(v)=}{} -((i\xi\cdot v)i + i v \cdot i\xi )-((i\xi\cdot v)\xi+\xi v\cdot i\xi)-\xi v \xi - i\xi \cdot v \cdot i\xi)\\
\hphantom{Q_\xi(v)}{}= \frac{1}{4}(2iv-2v+2i\xi\cdot v +\xi v-v\xi +2 \langle v,\xi \rangle i+2\langle v, i\xi \rangle\\
\hphantom{Q_\xi(v)=}{} -2 \langle\xi i,v \rangle i+2 \langle v,\xi \rangle i\xi+2\langle v,i\xi \rangle \xi+ 2\langle \xi,v\rangle \xi+2 \langle i\xi,v\rangle i\xi)\\
\hphantom{Q_\xi(v)}{}= \frac{1}{4}(2iv-2v+2i\xi\cdot v +2\xi v+2 \langle v,\xi \rangle +2 \langle v,\xi \rangle i+2\langle v, i\xi \rangle\\
\hphantom{Q_\xi(v)=}{} -2 \langle\xi i,v \rangle i+2 \langle v,\xi \rangle i\xi+2\langle v,i\xi \rangle \xi+ 2\langle \xi,v\rangle \xi+2 \langle i\xi,v\rangle i\xi)\\
\hphantom{Q_\xi(v)}{}= \frac{1}{2}(iv-v+i\xi\cdot v+\xi v+ (\langle v, \xi \rangle + \langle v, i\xi \rangle )(1+i+\xi+i\xi))\\
\hphantom{Q_\xi(v)}{}= \frac{1}{2}((-1+i+\xi+i\xi) v + \langle v, \xi + i\xi \rangle (1+i+\xi+i\xi)),
 \end{gather*}
 where in the sixth equality the formulas \eqref{eq:gxi1}, \eqref{eq:gxi2}, \eqref{eq:gxi3}, \eqref{eq:gxi4} and \eqref{eq:gxi5} were used, while in seventh equality the rule $v\xi=-\xi v-2 \langle v,\xi \rangle$ was applied.
\end{proof}

Using Lemma~\ref{lem:gxiv} the inverse function $Q_{\xi}^{-1}\colon V_\xi \rightarrow V_i$ can be calculated as well by observing that the roles of $i$ and $\xi$ are played by $-\xi$ and $-i$ respectively. Taking into account that any $v\in V_\xi$ is perpendicular to $\xi$, essentially the same calculation leads to
\begin{align*}
Q_\xi^{-1}(v) &=\bar{r}_\xi v r_\xi=\frac{1}{4}(1-i-\xi-i\xi)v(1+i+\xi+i\xi)\\
&=\frac{1}{4}(1+(-i)+(-\xi)+(-\xi)(-i))v(1-(-i)-(-\xi)-(-\xi)(-i))\\
&=(-1-\xi-i+\xi i) v + (\langle v, -i + \xi i \rangle )(1-\xi-i+\xi i).
\end{align*}
Moreover,
\begin{gather*}
Q_{-\xi}^{-1}(v)=(-1+\xi-i-\xi i) v + (\langle v, -i - \xi i \rangle )(1+\xi-i-\xi i).
\end{gather*}

\begin{Lemma}\label{lemma3.7}
 If $v, \xi \in V_i$, then
 \begin{gather*}
 Q_{-\xi}^{-1}\circ Q_\xi(v)=v \xi- \langle v \xi,1\rangle(1+\xi)- \langle v\xi ,i\rangle(1+\xi)i.
 \end{gather*}
\end{Lemma}
\begin{proof}
 To calculate $Q_{-\xi}^{-1}\circ Q_\xi(v)$ for an arbitrary $v \in V_i$ more preparation is needed.

(i) Applying Lemma~\ref{lem:octprop}(5) we obtain
 \begin{gather}
 \xi(i\xi \cdot v)+i\xi \cdot \xi v=\xi i \xi \cdot v + i\xi \xi \cdot v=iv-iv=0. \label{eq:gxginvst1}
 \end{gather}

(ii) By changing the order of terms in the multiplications one obtains
 \begin{gather*} i \cdot \xi v = v\xi \cdot i -2 \langle \xi v,i \rangle = - vi\cdot \xi - 2 \langle\xi v,i \rangle ,\\
 \xi \cdot iv=vi\cdot\xi-2\langle i v,\xi \rangle. \end{gather*}
 Using Lemma~\ref{lem:octprop}(6) and the definition of multiplication it can be proved, that
 \[ 2 \langle \xi v,i \rangle -2 \langle iv,\xi \rangle=4 \langle i\xi,v\rangle. \]
 Therefore,
 \begin{align}
 \xi \cdot iv - i\cdot \xi v&= 2 vi\cdot \xi + 2 \langle \xi v,i \rangle -2 \langle iv,\xi \rangle=2 vi\cdot \xi + 4 \langle i\xi,v\rangle \nonumber\\
 &=-2iv\cdot\xi+ 4 \langle i\xi,v\rangle=2i\xi \cdot v+4\langle v,\xi \rangle i+ 4 \langle i\xi,v\rangle , \label{eq:gxginvst}
 \end{align}
 and thus
 \begin{gather}
 -2 i\xi \cdot v+ \xi \cdot iv-i\cdot \xi v=4 \langle v, \xi \rangle i+4 \langle i\xi,v\rangle. \label{eq:gxginvst2}
 \end{gather}

(iii) By exchanging $\xi$ with $i\xi$ in \eqref{eq:gxginvst} one has
 \begin{gather}
 i\xi \cdot iv-i(i\xi \cdot v)=2 ii\xi v+4 \langle v, i\xi \rangle i+ 4 \langle ii\xi,v \rangle=-2 \xi v +4 \langle v, i\xi \rangle i- 4 \langle\xi,v\rangle .\label{eq:gxginvst3}
 \end{gather}

(iv) If $a,b \in \SO'$ and $a\perp b$, then $ab$ is orthogonal to both $a$ and $b$. Thus
 \begin{gather}
 \langle 1+i+\xi+i\xi ,-i+i\xi\rangle=0-1+0+1=0. \label{eq:gxginvst4}
 \end{gather}

(v) Finally, taking into account again the orthogonality assumptions and Lemma~\ref{lem:octprop}(6)
 \[ \langle i\xi\cdot v,i\rangle= - \langle \xi i\cdot v,i\rangle =\langle iv,\xi i \rangle+ 2\underbrace{ \langle i,\xi i \rangle \langle v,1 \rangle}_{0}=-\langle iv,i\xi \rangle=-\langle v,\xi \rangle.\]
 This leads to
 \begin{gather}
 \langle (-1+i+\xi+i\xi)v, -i+i\xi \rangle \nonumber \\
 \qquad{}= \underbrace{\langle -v, -i+i\xi \rangle}_{\langle -v, i\xi \rangle} + \underbrace{\langle iv,-i+i\xi \rangle}_{\langle iv,i\xi \rangle} + \langle \xi v,-i+i\xi \rangle+ \underbrace{\langle i\xi\cdot v, -i+i\xi \rangle}_{\langle i\xi\cdot v, -i \rangle} \nonumber\\
 \qquad{}= \langle -v, i\xi \rangle+\underbrace{\langle iv,i\xi \rangle}_{\langle v,\xi \rangle}-\underbrace{\langle \xi v,-i \rangle}_{-\langle v,\xi i \rangle}-\underbrace{\langle \xi v,\xi i \rangle}_{\langle v,i \rangle = 0}-(-\langle v,\xi \rangle)=2 \langle v, \xi i \rangle+2\langle v,\xi \rangle.
 \label{eq:gxginvst5}
 \end{gather}

 To simplify calculation it is useful to get rid of the constant factor. According to Lemma~\ref{lem:gxiv} we have
 \begin{gather*}
 4 Q_{-\xi}^{-1}\circ Q_\xi(v) = (-1-i+\xi+i\xi) ((-1+i+\xi+i\xi) v + \langle v, \xi + i\xi \rangle (1+i+\xi+i\xi) ) \\
\hphantom{4 Q_{-\xi}^{-1}\circ Q_\xi(v) =}{} + \langle (-1+i+\xi+i\xi) v + (\langle v, \xi + i\xi \rangle )(1+i+\xi+i\xi), -i - \xi i \rangle \\
\hphantom{4 Q_{-\xi}^{-1}\circ Q_\xi(v) =}{} \times (1-i+\xi+i\xi)\\
\hphantom{4 Q_{-\xi}^{-1}\circ Q_\xi(v)}{} = v-iv-\xi v-i\xi\cdot v+ \langle v, \xi + i\xi \rangle(-1-i-\xi-i\xi)\\
\hphantom{4 Q_{-\xi}^{-1}\circ Q_\xi(v) =}{} +iv-i^2v-i\cdot \xi v-i(i\xi\cdot v)+\langle v, \xi + i\xi \rangle\big({-}i-i^2-i\xi -i^2\xi\big)\\
\hphantom{4 Q_{-\xi}^{-1}\circ Q_\xi(v) =}{} -\xi v+\xi \cdot iv+\xi^2v+\xi (i\xi \cdot v)+\langle v, \xi + i\xi \rangle \big(\xi +\xi i+\xi^2+\xi i\xi\big)\\
\hphantom{4 Q_{-\xi}^{-1}\circ Q_\xi(v) =}{} -i\xi\cdot v+i\xi\cdot iv+i\xi\cdot \xi v+(i\xi)^2 v+ \langle v, \xi + i\xi \rangle \big(i\xi+i\xi i+i\xi^2+(i\xi)^2\big)\\
\hphantom{4 Q_{-\xi}^{-1}\circ Q_\xi(v) =}{} +\big[ \langle (-1+i+\xi+i\xi)v, -i+i\xi \rangle \\
 \hphantom{4 Q_{-\xi}^{-1}\circ Q_\xi(v) =}{} + \langle v, \xi + i\xi \rangle \langle 1+i+\xi+i\xi ,-i+i\xi\rangle \big] (1-i+\xi+i\xi)\\
\hphantom{4 Q_{-\xi}^{-1}\circ Q_\xi(v)}{} = -2\xi v+(-2 i\xi\cdot v+\xi\cdot iv-i\cdot \xi v) \\
\hphantom{4 Q_{-\xi}^{-1}\circ Q_\xi(v) =}{} +(i\xi \cdot iv-i(i\xi \cdot v))+(\xi (i\xi \cdot v)+i\xi\cdot \xi v)\\
\hphantom{4 Q_{-\xi}^{-1}\circ Q_\xi(v) =}{} +2\langle v, \xi + i\xi \rangle(-1-i+\xi+\xi i) \\
\hphantom{4 Q_{-\xi}^{-1}\circ Q_\xi(v) =}{} +\langle (-1+i+\xi+i\xi)v, -i+i\xi \rangle(1+\xi-i+i\xi)\\
\hphantom{4 Q_{-\xi}^{-1}\circ Q_\xi(v)}{} =-4\xi v-4\langle \xi,v\rangle+4\langle v, \xi \rangle i+4\langle i\xi,v \rangle+4 \langle v, i\xi \rangle i\\
\hphantom{4 Q_{-\xi}^{-1}\circ Q_\xi(v) =}{} +2(\langle v, \xi + i\xi \rangle)(-1-i+\xi+\xi i)+2(\langle v, \xi + i\xi \rangle)(1+\xi-i+i\xi )\\
\hphantom{4 Q_{-\xi}^{-1}\circ Q_\xi(v)}{} =-4\xi v+ \langle \xi,v\rangle (-4-2+2+4i-2i-2i+2\xi+2\xi+2\xi i+2i\xi)\\
\hphantom{4 Q_{-\xi}^{-1}\circ Q_\xi(v) =}{} +\langle i\xi,v\rangle(4-2-2+4i-2i+2i+2\xi-2\xi+2\xi i-2 i\xi)\\
\hphantom{4 Q_{-\xi}^{-1}\circ Q_\xi(v)}{} = -4\xi v+ \langle \xi,v\rangle(-4+4\xi)+ \langle i \xi,v\rangle(4i+4\xi i)\\
\hphantom{4 Q_{-\xi}^{-1}\circ Q_\xi(v) =}{} = 4v \xi+ \langle \xi,v\rangle(4+4\xi)+ \langle i \xi,v\rangle(4i+4\xi i),
\end{gather*}
 where in the fourth equality the formulas \eqref{eq:gxginvst1}, \eqref{eq:gxginvst2}, \eqref{eq:gxginvst3}, \eqref{eq:gxginvst4} and \eqref{eq:gxginvst5} were used.
 To sum it up, the required transformation is given by
\begin{align*}
 Q_{-\xi}^{-1}\circ Q_\xi(v)&= v \xi+ \langle \xi,v\rangle(1+\xi)+ \langle i \xi,v\rangle(1+\xi)i\\
 &=v \xi- \langle v \xi,1\rangle(1+\xi)- \langle v\xi ,i\rangle(1+\xi)i.\tag*{\qed}
 \end{align*}\renewcommand{\qed}{}
\end{proof}

\begin{proof}[Proof of Theorem~\ref{thm:traneq}]
 As mentioned earlier, the subspace $V_i$ is a complex linear space with basis $j$, $e$, $g$ and complex structure $J_i\colon V_i \rightarrow V_i$, $v \mapsto iv$. Since $\xi \in V_i$, the coordinate expression of $\xi$ in $V_i$ can be written as
 \begin{align*} \xi = uj+ve+wg&=(u_1+u_2 I)j+(v_1+v_2 I)e+(w_1+w_2 I)g\\
 & = u_1j+u_2k+v_1e+v_2f+w_1g-w_2h,
 \end{align*}
 where $u,v,w \in \SC$, $u_i,v_i,w_i \in \SR$ for $i=1,2$, and $I$ is again the imaginary unit in the field $\SC$.
 Because $V_i=V_{-i}$ as a subspace, $\xi$ can be expressed as a element of $V_{-i}$ as well. Here the basis is the same, but the complex structure is given by $J_{-i}\colon V_{-i} \rightarrow V_{-i}$, $v \mapsto -iv$. Therefore, the coordinate expression of the same $\xi$ here is
 \[ \xi =\overline{u}j+\overline{v}e+\overline{w}g=(u_1-u_2 I)j+(v_1-v_2 I)e+(w_1-w_2 I)g .\]
 According to the multiplication rule of the basis vectors of $\SO$ it is possible to compute the multiplication of $\xi$ with the basis vectors from the left as
 \begin{gather*}
 j\xi = -u_1+u_2i+v_1g+v_2h-w_1e+w_2f=-u\cdot 1+0j-we+vg,\\
 e\xi =-u_1g-u_2h-v_1+v_2i+w_1j-w_2k=-v\cdot 1+wj+0e-ug,\\
 g\xi =u_1e-u_2f-v_1j+v_2k-w_1+w_2i=-w\cdot 1-vj+ue+0g,
 \end{gather*}
 because the resulting vector $v$, of which the terms are calculated here, is in $V_{-i}$. Similarly,
 \begin{align*}
 \xi i &= -u_1k+u_2j-v_1f+v_2e+w_1h+w_2g\\
 &=(u_2+ u_1 I)j+(v_2+v_1 I)e+(w_2+w_1 I)g\\
 &=(\overline{u} I)j+(\overline{v} I)e+(\overline{w}I)g .
 \end{align*}

Using Lemma~\ref{lemma3.7} we get
 \begin{gather*}
 Q_{-\xi}^{-1}\circ Q_\xi(j)=j \xi- \langle j \xi,1\rangle(1+\xi)- \langle j\xi ,i\rangle (1+\xi)i \\
 \qquad{}=\begin{pmatrix}
 0 \\ -w \\ v
 \end{pmatrix} + u_1
 \begin{pmatrix}
 \overline{u} \\ \overline{v} \\ \overline{w}
 \end{pmatrix}
 -u_2
 \begin{pmatrix}
 \overline{u} I \\ \overline{v} I \\ \overline{w}I
 \end{pmatrix}=
 \begin{pmatrix}
 0 \\ -w \\ v
 \end{pmatrix} + \underbrace{(u_1-u_2I)}_{\overline{u}}
 \begin{pmatrix}
 \overline{u} \\ \overline{v} \\ \overline{w}
 \end{pmatrix}=
 \begin{pmatrix}
 \overline{u}^2 \\ \overline{u}\overline{v}-w \\ \overline{u}\overline{w}+v
 \end{pmatrix},\\
 Q_{-\xi}^{-1}\circ Q_\xi(e)=e \xi- \langle e \xi,1\rangle(1+\xi)- \langle e\xi ,i\rangle (1+\xi)i \\
 \qquad{}=\begin{pmatrix}
 w \\ 0 \\ -u
 \end{pmatrix} + v_1
 \begin{pmatrix}
 \overline{u} \\ \overline{v} \\ \overline{w}
 \end{pmatrix}
 - v_2
 \begin{pmatrix}
 \overline{u} I \\ \overline{v} I \\ \overline{w}I
 \end{pmatrix}=
 \begin{pmatrix}
 w \\ 0 \\ -u
 \end{pmatrix}+ \underbrace{(v_1-v_2I)}_{\overline{v}}
 \begin{pmatrix}
 \overline{u} \\ \overline{v} \\ \overline{w}
 \end{pmatrix}=
 \begin{pmatrix}
 \overline{vu} +w\\ \overline{v}^2 \\ \overline{vw}-u
 \end{pmatrix},\\
 Q_{-\xi}^{-1}\circ Q_\xi(g)=g \xi- \langle g \xi,1\rangle(1+\xi)- \langle g\xi ,i\rangle (1+\xi)i \\
 \qquad{}=\begin{pmatrix}
 -v \\ u \\ 0
 \end{pmatrix} + w_1
 \begin{pmatrix}
 \overline{u} \\ \overline{v} \\ \overline{w}
 \end{pmatrix}
 - w_2
 \begin{pmatrix}
 \overline{u} I \\ \overline{v} I \\ \overline{w}I
 \end{pmatrix}=
 \begin{pmatrix}
 -v \\ u \\ 0
 \end{pmatrix}+ \underbrace{(w_1-w_2I)}_{\overline{w}}
 \begin{pmatrix}
 \overline{u} \\ \overline{v} \\ \overline{w}
 \end{pmatrix}=
 \begin{pmatrix}
 \overline{wu} -v\\ \overline{wv}+u \\ \overline{w}^2
 \end{pmatrix}.
 \end{gather*}

 Putting all together, the matrix which represents the mapping $Q_{-\xi}^{-1}\circ Q_\xi\colon V_i \rightarrow V_{-i}$ is
 \[
 M_{\overline{\xi}}=
 \begin{pmatrix}
 \overline{u}^2 & \overline{vu} +w &\overline{wu} -v\\
 \overline{u}\overline{v}-w & \overline{v}^2 & \overline{wv}+u \\
 \overline{u}\overline{w}+v& \overline{vw}-u & \overline{w}^2
 \end{pmatrix},
 \]
 and to get matrix of the same function as a $V_{-i} \rightarrow V_{-i}$ mapping each complex coordinate of $\xi$ should be conjugated:
 \[
 M_\xi=
 \begin{pmatrix}
 u^2 & vu +\overline{w} &wu-\overline{v}\\
 uv-\overline{w} & v^2 & wv+\overline{u} \\
 uw+\overline{v}& vw-\overline{u} & w^2
 \end{pmatrix}.
 \]
 This proves the statement.
\end{proof}

\subsection[The class of $G_2$]{The class of $\boldsymbol{G_2}$}\label{section3.3}

As it is known the principal ${\rm SU}(3)$-bundles over $S^6$ are classified by $\pi_5({\rm SU}(3))$. The following fact is well known, but again we included a sketch proof of it.
\begin{Proposition} $\pi_5({\rm SU}(3)) =\SZ $.
\end{Proposition}
\begin{proof}[Sketch proof] From the well-known periodicity theorem of Bott~\cite{bott1959stable} it follows that $\pi_5({\rm SU}(4))\allowbreak =\SZ$.
 It can be shown as well that ${\rm SU}(4)={\rm Spin}(6)$. By definition ${\rm Spin}(6)$ is the double cover of~${\rm SO}(6)$. A covering mapping induces isomorphisms on the higher homotopy groups of the total and base spaces. Thus, $\pi_5({\rm Spin}(6))=\pi_5({\rm SO}(6))$. Moreover, $\SC{\rm P}^3={\rm SO}(6)/{\rm U}(3)$ and from the long exact sequence of this fibration one obtains $\pi_5({\rm SO}(6))=\pi_5({\rm U}(3))$. Finally the mapping $\det \colon {\rm U}(3) \rightarrow {\rm U}(1)$ is a locally trivial fibration with fibers $\det^{-1}(1)={\rm SU}(3)$. From the long exact sequence of this fibration one obtains $\pi_5({\rm U}(3))=\pi_5({\rm SU}(3))$.
\end{proof}

Our proof of the next statement is an adaptation of \cite[Proposition~2]{chaves96complex}. It provides Theorem~\ref{thm:transfun}.
\begin{Proposition} The map $\theta\colon S^5 \rightarrow {\rm SU}(3)$ from Theorem~{\rm \ref{thm:traneq}} is the generator of $\pi_{5}({\rm SU}(3))$.
\end{Proposition}
\begin{proof} The columns of a matrix in ${\rm SU}(3)$ are unit length vectors in $\SC^3$. Define a mapping $\pi\colon {\rm SU}(3) \rightarrow S^5$ as the projection onto the first column. Then the fiber above, e.g., $(1,0,0)$ is ${\rm SU}(2)$ and therefore $\pi\colon {\rm SU}(3) \rightarrow S^5$ is a fibration with fibers ${\rm SU}(2)$. Then the long exact homotopy sequence of this fibration gives
 \[ \underbrace{\pi_5({\rm SU}(3))}_{\SZ} \overset{\pi_\ast}{\longrightarrow} \underbrace{\pi_5\big(S^5\big)}_{\SZ} \longrightarrow \underbrace{\pi_4({\rm SU}(2))}_{\SZ_2} \longrightarrow \underbrace{\pi_4({\rm SU}(3))}_{0}.\]
 Because the mapping $\pi_4({\rm SU}(2)) \rightarrow \pi_4({\rm SU}(3))$ is surjective, the map $\pi_\ast$ should be multiplication by 2. A generator of $\pi_5\big(S^5\big)$ is just a map $S^5 \rightarrow S^5$ of degree one. The degree of $\pi \circ \theta\colon S^5 \rightarrow S^5$ is 2, because this mapping is just the first column of $\theta$. For example, the point $(1,0,0)$ has preimage $\{(1,0,0),(-1,0,0)\}$. It can be checked that the corresponding signs are the same and therefore $\pi_\ast([\theta])=2$. Thus $[\theta]$ is a generator of $\pi_{5}({\rm SU}(3))$.
\end{proof}

\subsection*{Acknowledgements}
The main part of the work was carried out while the author was at the Budapest University of Technology and Economics, Hungary. The author would like to thank to G\'abor Etesi and to Szil\'ard Szab\'o for several helpful comments and discussions. The author is also thankful to the anonymous referees.

\pdfbookmark[1]{References}{ref}
\LastPageEnding


\begin{thebibliography}{99}
\footnotesize\itemsep=0pt

\bibitem{baez2002octonions}
Baez J.C., The octonions, \href{https://doi.org/10.1090/S0273-0979-01-00934-X}{\textit{Bull. Amer. Math. Soc.}} \textbf{39} (2002),
 145--205, \href{https://arxiv.org/abs/math.RA/0105155}{arXiv:math.RA/0105155}.

\bibitem{bott1959stable}
Bott R., The stable homotopy of the classical groups, \href{https://doi.org/10.2307/1970106}{\textit{Ann. of Math.}}
 \textbf{70} (1959), 313--337.

\bibitem{chaves96complex}
Chaves L.M., Rigas A., Complex reflections and polynomial generators of
 homotopy groups, \textit{J.~Lie Theory} \textbf{6} (1996), 19--22.

\bibitem{ehresmann1950varietes}
Ehresmann C., Sur les vari\'{e}t\'{e}s presque complexes, in Proceedings of the
 {I}nternational {C}ongress of {M}athematicians, {C}ambridge, {M}ass., 1950,
 Vol.~2, Amer. Math. Soc., Providence, R.I., 1952, 412--419.

\bibitem{lamont1963arithmetics}
Lamont P.J.C., Arithmetics in {C}ayley's algebra, \href{https://doi.org/10.1017/S2040618500034808}{\textit{Proc. Glasgow Math.
 Assoc.}} \textbf{6} (1963), 99--106.

\bibitem{lee2003introduction}
Lee J.M., Introduction to smooth manifolds, \textit{Graduate Texts in
 Mathematics}, Vol.~218, \href{https://doi.org/10.1007/978-0-387-21752-9}{Springer-Verlag}, New York, 2003.

\bibitem{pirisi2017motivic}
Pirisi R., Talpo M., On the motivic class of the classifying stack of {$G_2$}
 and the spin groups, \href{https://doi.org/10.1093/imrn/rnx208}{\textit{Int. Math. Res. Not.}} \textbf{2019} (2019),
 3265--3298, \href{https://arxiv.org/abs/1702.02649}{arXiv:1702.02649}.

\bibitem{postnikov1986lie}
Postnikov M., Lectures in geometry. Semester V: Lie groups and {L}ie algebras,
 Mir, Moscow, 1986.

\bibitem{puttmann2003presentations}
P\"{u}ttmann T., Rigas A., Presentations of the first homotopy groups of the
 unitary groups, \href{https://doi.org/10.1007/s00014-003-0770-0}{\textit{Comment. Math. Helv.}} \textbf{78} (2003), 648--662,
 \href{https://arxiv.org/abs/math.AT/0301192}{arXiv:math.AT/0301192}.

\end{thebibliography}
\end{document}